\documentclass{amsart}
\usepackage[english]{babel}
\usepackage[utf8]{inputenc}
\usepackage{amsmath, amssymb,epic,graphicx,mathrsfs,enumerate}
\usepackage[all]{xy}
\usepackage{color}
\usepackage{comment}
\usepackage{enumitem}
\usepackage{hyperref}
\usepackage[]{frontespizio}

\usepackage{amsmath}
\usepackage{amsthm}
\usepackage{amssymb}
\usepackage{latexsym}
\usepackage{epsfig}

\DeclareMathOperator{\perm}{Sym}
 
\DeclareMathOperator{\aut}{Aut} 
\DeclareMathOperator{\soc}{soc}

\DeclareMathOperator{\oo}{O}

\DeclareMathOperator{\kndo}{End}
\DeclareMathOperator{\gl}{GL} 
\DeclareMathOperator{\syl}{Syl}

\DeclareMathOperator{\R}{R}

\DeclareMathOperator{\frat}{Frat}

\newcommand{\gen}[1]{\left\langle#1\right\rangle} 
\newcommand{\al}{\alpha}

\newcommand{\st}{such that }
\newcommand{\ifa}{if and only if }

\newcommand{\gr}{group }

\newcommand{\w}[1]{\widetilde{#1}}

\newcommand{\fff}{\mathfrak{F}_2}
\newcommand{\ccc}{\mathfrak{C}}

\newcommand{\ff}{\mathfrak{F}}
\newcommand{\ddd}{\mathfrak{D}}

\newcommand{\G}{\Gamma}

\newcommand{\nn}{\mathrel{\unlhd}}

\newcommand{\gf}{\G_\ff}
\newcommand{\isod}{\mathcal{I}_\dd}
\newcommand{\iso}{\mathcal{I}_\ff}
\newcommand{\isoc}{\mathcal{I}_\ccc}
\newcommand{\isoo}{\mathcal{I}_{\fff}}

\newcommand{\dd}{\mathfrak{D}}

\newtheorem{thm}{Theorem}
\newtheorem{cor}[thm]{Corollary}
 \newtheorem{lemma}[thm]{Lemma}
\newtheorem{prop}[thm]{Proposition} \newtheorem{rem}[thm]{Remark}
 \newtheorem{defn}[thm]{Definition}

\numberwithin{equation}{section}

\renewcommand{\footnote}{\endnote}
\newcommand{\ignore}[1]{}\makeglossary

\begin{document}
	\bibliographystyle{amsplain}
\title[Semiregularity and connectivity of the non-$\mathfrak F$ graph of a finite group]{Semiregularity and connectivity\\ of the non-$\mathfrak F$ graph of a finite group}

\author{Andrea Lucchini}
\address{Andrea Lucchini\\ Universit\`a di Padova\\  Dipartimento di Matematica \lq\lq Tullio Levi-Civita\rq\rq\\ Via Trieste 63, 35121 Padova, Italy\\email: lucchini@math.unipd.it}
\author{Daniele Nemmi}
\address{Daniele Nemmi\\ Universit\`a di Padova\\  Dipartimento di Matematica \lq\lq Tullio Levi-Civita\rq\rq\\ Via Trieste 63, 35121 Padova, Italy\\email: dnemmi@math.unipd.it}


\begin{abstract}Given a class $\mathfrak F$ of finite groups, we consider the graph $\w{\G}_\ff(G)$  whose vertices are the elements of $G$ and where two  vertices $g,h\in G$  are adjacent \ifa $\gen{g,h}\notin\ff$. Moreover we denote by $\iso(G)$ the set of the isolated vertices of $\w{\G}_\ff(G).$ 
	 We  address the following question: 
	to what extent the fact that $\iso(H)$ is a subgroup of $H$
	for any $H\leq G,$ implies that the graph $\gf(G)$ obtained from $\w{\G}_\ff(G)$ 
 by deleting the isolated vertices is a connected graph?
\end{abstract}
\maketitle

\section{Introduction}
Let $\mathfrak F$ be a class of finite groups and $G$ a finite group. Consider the graph $\w{\G}_\ff(G)$  whose vertices are the elements of $G$ and where two  vertices $g,h\in G$  are adjacent \ifa $\gen{g,h}\notin\ff$. Moreover  denote by $\iso(G)$ the set of isolated vertices of $\w{\G}_\ff(G)$. In \cite{nof} we defined  the non-$\ff$ graph $\gf(G)$ of $G$ as the subgraph of $\w{\G}_\ff(G)$ obtained by deleting the isolated vertices.  In that paper we concentrated our attention in the particular case when $\ff$ is a saturated formation. In particular we addressed the question whether $\iso(G)$ is a subgroup of $G$. This is not always the case, however it occurs for several saturated formations and, more in general, we called semiregular a class $\ff$ with the property
that $\iso(G)$ is a subgroup of $G$ for every finite group $G$.
In the same paper we called connected a class $\ff$ with the property that the graph $\gf(G)$ is connected for any finite group $G$. The results obtained in \cite{nof} indicate that often a semiregular formation is connected. This occurs for example for the formations of abelian groups, nilpotent groups, soluble groups, supersoluble groups, groups with nilpotent derived subgroup, groups with Fitting length less or equal then $t$ for any $t\in \mathbb N.$

This stimulated us to investigate to what extent the semiregularity of $\ff$ implies its connectivity. We addressed the following more general question. Suppose that $\ff$ is a class containing only soluble groups and closed under taking subgroups and that a finite group $G$ has the property that $\iso(H)$ is a subgroup of $H$ for any $H\leq G$. Does this implies that $\gf(G)$ is connected?

Our main theorem (Theorem \ref{main}) says that for a fixed class $\ff$, either the answer is affirmative, or a minimal counterexample $G$ has a very peculiar behaviour. Indeed $G$ is soluble, it cannot be generated with 2 elements and
there exists an epimorphism 
$\pi: G \to  V^t \rtimes H,$ where
$H$ is 2-generated, $V$ is a faithful irreducible $H$-module,  $t=1+\dim_{\kndo_H(V)}(V),$ such that the following properties hold.
Denote by $\mathcal W$ be the set of the $H$-submodules  of $V^t$ that are $H$-isomorphic to $V^{t-1},$ 
and for any $W \in \mathcal W$, let $M_W=
\pi^{-1}(WH).$ Then there exists $W \in \mathcal W$ such that if $\overline W \neq W,$ then
there is no edge in $\gf(G)$ connecting two elements of $M_{\overline W}$ and one of  the two following situations occurs:
\begin{enumerate}
	\item any egde in $\gf(G)$ belongs to the subgraph induced by a conjugate of $M_W;$
	\item $H$ is cyclic of prime order and any egde in $\gf(G)$ belongs either to the subgraph induced by a conjugate of $M_W$ or to the subgraph induced by $N=\pi^{-1}(V^t).$
\end{enumerate}
We construct two examples in which these two situations occur.
In the first example, $H$ is a quaternion group, $V\cong C_p \times C_p$, where $p$ is an arbitrary odd prime, and $\gf(G)$ has $p$ connected components, corresponding to the non-isolated vertices in the subgraphs of $\gf(G)$ induced by the $p$ different conjugates of $M_W$. In the second example, $H\cong C_3$, $V\cong C_2\times C_2$   and $\gf(G)$ has two connected components: one consisting of the non-isolated vertices of $\gf(N),$ the other one consisting of the union of the sets of non-isolated vertices in the four conjugates of $M_W.$

Despite the presence of these examples, Theorem \ref{main} can be used to prove that the following classes of finite groups are connected: the class of cyclic finite groups, the class of finite groups whose order is divisible by only one prime, the class of finite groups whose order is divisible by at most two primes. 
\hbox{}

\bibliographystyle{alpha}

\section{Proof of the main theorem}

We begin this section by recalling some known results and proving some preliminary lemmas, needed in the proof.
\begin{prop}\cite{Ga}\label{gaz}
Let $N$ be a normal
subgroup of a finite group $G$ and suppose that $\langle g_1,\dots g_k\rangle N=G$. If $k \geq d(G),$ then there
exist $n_1,\dots,n_k \in N$ so that $\langle g_1n_1,\dots g_kn_k\rangle=G$.
\end{prop}

Recall that the generating graph of a finite group $G$ is the graph whose vertices are the elements of $G$ and where $g_1$ and $g_2$ are adjacent if and only if $\langle g_1, g_2 \rangle=G.$
	
\begin{lemma}\label{unodue}
	Let $G$ be a 2-generated finite soluble group and denote by $\Omega(G)$ the set of non-isolated vertices of the generating graph of $G$. Assume that $N$ is a minimal normal subgroup of $G$ and that $g_1, g_2$ are two elements of $G$ with $G=\langle g_1, g_2, N\rangle.$ Then there exists $i\in \{1,2\}$ with the property that $g_iN\subseteq \Omega(G).$
\end{lemma}
\begin{proof}
	Suppose $g_1n^* \notin \Omega(G)$ for some $n^* \in N.$ If $n\in N,$ then $\langle g_1n^*, g_2n, N\rangle=G$, while $\langle g_1n^*, g_2n
	\rangle<G.$ By \cite[Lemma 1]{symb}, there exists $n^\prime\in N$ with $\langle g_1n^\prime, g_2n
	\rangle=G.$ Hence $g_2n\in \Omega(G).$
\end{proof}

\begin{defn}Let $\ff$ be a class of finite groups. We define $\ff_2$ as the class of the finite groups $G$ with the property that any 2-generated subgroup of $G$ is in $\ff.$
\end{defn}

Assume that $\ff$ is closed under taking subgroups. We say that $\ff$ is 2-recognizable whenever a group $G$ belongs to $\ff$ if all 2-generated subgroups of $G$ belong
to $F$. If $\ff$ is not 2-recognizable, then $\ff_2 > \ff.$ For example, if $\ff$ is the class of the metabelian finite groups, and $G$ is a Sylow 2-subgroup of $\perm(8)$, then $G\in \ff_2 \setminus \ff.$
However the following can be immediately seen.
\begin{lemma}
If $G$ is a finite group, then $\gf(G)=\G_{\ff_2}(G).$ In particular $\iso(G)=\isoo(G).$
\end{lemma}
\begin{lemma}\label{Icyc}
Let $\ff$ be a class of finite groups. If $\iso(G)$ is a subgroup of $G$ and $G=\iso(G)\!\gen{g}$ for some $g\in G$, then $G\in\fff$.
\end{lemma}
\begin{proof}
	Let $x$ be an arbitrary element of $G$. We have $x=ig^\al$ for some $i\in\iso(G)$ and $\al\in\mathbb{Z}$. Moreover $\gen{g,ig^\al}=\gen{g,i}\in\ff$, since $i\in\iso(G)$. Hence $g\in\iso(G)=\isoo(G)$, so $G=\isoo(G)$ and therefore $G\in\fff$.
\end{proof}

\begin{thm}\cite[Theorem 16.6]{dh}\label{masu}
	Let $L$ and $M$ be inconjugate maximal subgroups of a finite soluble group $G.$ If $M_G \not\leq L_G,$ then $L\cap M$ is a maximal subgroup of $L.$
\end{thm}

\begin{lemma}\label{duemax}
	Let $G$ be a finite soluble group. Suppose that $M$ and $L$ are maximal subgroups of $G$ with $M, L \notin \fff$ and that
	the graphs $\gf(L)$ and $\gf(M)$ are connected. Denote by $\Gamma_M$ (resp. $\Gamma_L$) the connected component of
	$\gf(G)$ containing the  vertices
	of $\gf(M)$ (resp. $\gf(L)$).
	\begin{enumerate}[label=(\alph*)]
		\item 
		If $L_G\not\leq M_G,$ then either $\Gamma_L=\Gamma_M$ or $L\cap M \subseteq \iso(L).$ 
		\item If $M_G\not\leq L_G$ and $L_G\not\leq M_G,$ then $\Gamma_L=\Gamma_M.$
		\item If $M_G < L_G$ and $\Gamma_L \neq \Gamma_M,$ then $L_G/M_G$ is the unique minimal normal subgroup of $G/M_G.$ 
		\item If $M_G < L_G$ and $\Gamma_L \neq \Gamma_M,$ then
		either $L$ is normal in $G$ or $L/L_G$ is cyclic.
	\end{enumerate}
\end{lemma}
\begin{proof} (a)
	If $\Gamma_M \neq \Gamma_L,$ then $L\cap M \subseteq \iso(M) \cup \iso(L).$ Thus either $L\cap M \leq \iso(L)$ or $L\cap M \leq \iso(M)$.
	In the second case, by Theorem \ref{masu},
	$L\cap M$ is a a maximal subgroup of $M,$
	hence $M=\langle L\cap M, m \rangle$ for some $m \in M.$ But then $M=\iso(M)\langle m\rangle.$ By Lemma \ref{Icyc}, $M \in \fff,$ against the hypotheses. 
	
	\noindent (b) If $\Gamma_M\neq \Gamma_L,$ then, by (a), $L\cap M \leq\iso(L) \cap \iso(M)$, but then $L, M \in \fff$, against the hypotheses.
	
	\noindent (c) Since $L_G \not\leq M_G,$ by (b) we have $K:=L\cap M\leq \iso(L).$ Let $N/M_G := \soc(G/M_G).$ Then $N \leq L_G,$ since
	$N/M_G$ is the unique minimal normal subgroup of $G/M_G$.  In particular $L=L\cap MN=(L\cap M)N=KN$, and therefore $K^L=K^N.$ If $L_G\cap M > M_G,$ then $N \leq (L_G\cap M)^N \leq K^N=K^L,$
	and therefore $L=KN=K^L\leq\iso(L)$ and $L \in \ff_2,$ against the assumptions. Hence $L_G \cap M=M_G,$ and consequently $L_G=N.$
	
	\noindent (d) Suppose $L$ is not normal in $G$. Let $U/L_G:=\soc(G/L_G)$. As in the proof of (c), we have $L=(L\cap M)N=(L\cap M)L_G.$ Since $L\notin\ff_2$ but, by (a) $L\cap M\leq\iso(L)$, there exist $l\in L_G\setminus\iso(L)$ and $s\in L$ \st $\gen{l,s}\notin\ff$.
	If $U\!\gen{s}<G$, then $U\!\gen{s}\leq S<G$ for a maximal subgroup $S$ with $L_G < U\leq S_G$. Since $\gen{l,s}\leq L\cap S$, we have $\G_L=\G_S$ and from (c), $\G_M=\G_S=\G_L$, against the assumptions. Therefore $U\!\gen{s}=G$ and $G/U\cong L/L_G$ is cyclic.
\end{proof}

The following is immediate.
\begin{lemma}\label{FFF} Let $\ff$ be a class of groups which is closed under taking subgroups and epimorphic images. Let $g,h\in G$ and $N\nn G$.
	\begin{enumerate}[label=(\alph*)]
		\item If $gN$ and $hN$ are adjacent vertices of $\gf(G/N)$, then $g$ and $h$ are adjacent vertices of $\gf(G)$.
		\item If $g\in \iso(G),$ then $gN\in \iso(G/N).$
		\item $\iso(G)^\sigma=\iso(G)$ for every $\sigma\in\aut(G)$.
	\end{enumerate}
\end{lemma}

\begin{defn}Let $\ff$ be a class of finite groups. We say that a finite group $G$ is $\ff$-semiregular if $\iso(H)\leq H$ for every $H\leq G.$
\end{defn}

\begin{lemma}\label{normale}
	Let $\ff$ be a class of finite groups with the following properties:
	\begin{enumerate}
		\item All the groups in $\ff$ are soluble.
		\item $\ff$ is closed under taking subgroups.
	\end{enumerate}
	Suppose that $G$ is  $\ff$-semiregular 
	 and $\gf(H)$ is connected for any proper subgroup 
	$H$ of $G.$ 
	If there exists a proper normal subgroup $N$ of $G$ such that
	$G\setminus N$ is contained in a unique connected component of $\gf(G),$ then either $\gf(G)$ is connected or $N$ is a maximal subgroup of $G.$
\end{lemma}

\begin{proof}
	Let $\Omega$ be the connected component of $\gf(G)$ containing $G\setminus N.$

	If $N\in\ff_2$, then every element of $N\setminus \iso(G)$ must be adjacent to an element of $G\setminus N,$ so $N\setminus \iso(G) \subseteq \Omega$. But this implies $\Omega=G\setminus \iso(G)$, and consequently $\gf(G)$ is connected.
	
	If  $N \notin\ff_2$, then $\iso(N)<N.$ Let $H$ be a maximal subgroup of $G$ containing $N$. Since
	$\gf(H)$ is connected, there exists a unique connected component of
	$\gf(G),$ say $\Delta,$ containing $H\setminus \iso(H).$ Of course $N\setminus \iso(N) \subseteq H\setminus \iso(H),$ so $N\setminus \iso(N)\subseteq \Delta.$ Recall that $G\setminus N\subseteq \Omega.$ Moreover if $x\in \iso(N) \setminus \iso(G),$ then $\langle x, y \rangle \notin \ff$ for some $y\in G\setminus N,$ so $\iso(N) \setminus \iso(G)\subseteq \Omega$. If $\Delta\cap \Omega\neq \emptyset,$ then $\Delta=\Omega=G\setminus \iso(G)$ and $\gf(G)$ is connected. If $\Delta\cap \Omega=\emptyset,$ then
	$(H\setminus \iso(H))\cap (H\setminus N)=\emptyset$, i.e. $H=N\cup \iso(H).$ Since $H\notin \fff$, $\iso(H)\neq H,$ and consequently $H=N.$ 
\end{proof}

\begin{lemma}
	Let $\ff$ be a class of finite groups with the following properties:
	\begin{enumerate}
		\item All the groups in $\ff$ are soluble.
		\item $\ff$ is closed under taking subgroups.
	\end{enumerate}
	Suppose that $G$ is  $\ff$-semiregular
	 and $\gf(H)$ is connected for any proper subgroup 
	$H$ of $G.$ 
	If $G$ is not soluble, then $\gf(G)$ is connected.
\end{lemma}
\begin{proof}
	Let $R=\R(G)$ be the soluble radical of $G$ and let $\mathfrak S$ be the class of the finite soluble groups.	Assume that $G\neq \R(G).$ 
	By \cite[Theorem 6.4]{gu}, the only isolated vertex of $\tilde \Gamma_{\mathfrak S}(G/R)$ is the identity element and the graph 	
	$\Gamma_{\mathfrak S}(G/R)$ is connected. By Lemma \ref{FFF}, all the elements of $G\setminus R$ belong to the same connected component of  $\Gamma_{\mathfrak S}(G)$. Since $\ff \subseteq \mathfrak S,$ there exists a connected component, say $\Omega$, of $\gf(G)$ with
	$G\setminus R \subseteq \Omega.$	
	Since $R$ cannot be a maximal subgroup of $G$, it follows from  Lemma \ref{normale}  that $\gf(G)$ is connected.
\end{proof}

\begin{lemma}\label{duegen}
Let $\ff$ be a class of finite groups with the following properties:
\begin{enumerate}
	\item All the groups in $\ff$ are soluble.
	\item $\ff$ is closed under taking subgroups.
\end{enumerate}
Suppose that $G$ is  $\ff$-semiregular 
and $\gf(H)$ is connected for any proper subgroup 
$H$ of $G.$ 
If $G$ is a 2-generated  group, then $\gf(G)$ is connected.
\end{lemma}
\begin{proof}By the previous lemma we may assume that $G$ is soluble. Moreover we may assume that if $g \in G,$ then $\langle g \rangle \in \ff$ (otherwise $g$ would be a universal vertex of $\gf(G)$ and consequently $\gf(G)$ would be obviously connected).
Finally, we  may assume $G\notin \ff_2$ (otherwise $\gf(G)$ is the empty graph). This implies that the subgraph $\Delta(G)$ of the generating graph of $G$ induced by its non-isolated vertices is a subgraph of $\gf(G).$ By \cite[Theorem 1]{cl}, the graph $\Delta(G)$ is connected, so there exists a connected component, say $\Gamma$, of $\gf(G)$ containing the set $\Omega(G)$ of the
	vertices of $\Delta(G).$ From now on, for any subgroup $H$ of $G,$ we will denote by $\Omega_\ff(H)$ the set $H\setminus \iso(H)$
	of the vertices of $\gf(H).$
	Since $\gf(G)$ is disconnected, we must have $\Omega_\ff(G)
	\neq \Gamma.$
	
	Let $g \in \Omega_\ff(G) \setminus \Gamma.$ Then there exists $\tilde g \in G$ such that $\langle g, \tilde g \rangle \notin \ff.$ Since $g\notin \Omega(G)$, $\langle g, \tilde g \rangle \neq G$ and therefore there exists a maximal subgroup $M$ of $G$ with $g\in \Omega_\ff(M).$ 
	We have
	\begin{equation}\label{inte}
	\Omega_\ff(M) \cap \Omega(G)=\emptyset.
	\end{equation}
	Indeed assume $m \in \Omega_\ff(M) \cap \Omega(G)$. Since  $m, g\in \Omega_\ff(M),$ then $g$ and $m$ would be contained in the same connected component of $\gf(M),$  and consequently in the same connected component of $\gf(G).$ However, since $m \in \Omega,$ this connected component would coincide with $\Gamma,$ in contradiction with $g \notin \Gamma.$

	We distinguish two cases:
	
	\noindent a)  $M$ is a normal subgroup of $G$. In this case
	$\iso(M)$ is a characteristic subgroup of $M$ and therefore it is normal in $G.$ Moreover $G=\langle y, M\rangle$ for some $y\in G.$ Since $G$ is 2-generated, by Proposition \ref{gaz}, there exist $m_1, m_2 \in M$ such that $G=\langle ym_1, m_2\rangle.$  Since $\Omega_\ff(M)\neq \emptyset$, we must have $M \notin \fff$ and therefore, by Lemma \ref{Icyc}, $M/\iso(M)$ is not cyclic. This implies $m_2\notin \iso(M).$ Hence $m_2 \in \Omega_\ff(M) \cap \Omega(G),$ against (\ref{inte}).
	
	\noindent b) $M$ is not a normal subgroup of $G$.	Let $M_G$ be the normal core of $M$ in $G$. We have $G/M_G\cong A/M_G \rtimes M/M_G,$ where $A/M_G$ is a faithful irreducible $(M/M_G)$-module. By \cite[Theorem 7]{dlt} there exist $k_1, k_2 \in M$ and $a \in A$ such that
	$G=\langle k_1, k_2^a\rangle M_G.$ By Proposition 
	\ref{gaz}, there exist $m_1, m_2 \in M_G$ 
	such that $G=\langle k_1m_1, k_2^am_2\rangle.$ Hence $$G=\langle x, y^a \rangle \text{ with }x=k_1m_1, y=k_2m_2^{a^{-1}}
	\in M \text { and } a \in A.$$ Now we claim that $M_G$ is contained in $\iso(M)$. Consider a normal series
	$$1=U_0\leq U_1 \leq \dots \leq U_t=M_G$$
	where, for $1\leq i \leq t,$ $U_i/U_{i-1}$ is a chief factor of $G.$ We prove by induction on $i$ that $U_i \leq \iso(M).$
	Since $\ff$ contains all  cyclic subgroups of $G$, $U_0=1 \leq \iso(M)$. Let $i<t$ and assume $U_i\leq \iso(M)$. It follows from Lemma \ref{unodue} that there are two possibilities:
	either  $xuU_i \in \Omega(G/U_i)$ for any $u \in U_{i+1}$ or $y^auU_i \in \Omega(G/U_i)$ for any $u \in U_{i+1}.$ In the first case, by Proposition \ref{gaz}, for any $u\in U_{i+1}$ there exists $\bar u \in U_i$ such that $xu\bar u \in \Omega(G).$ Hence, by (\ref{inte}), $xu\bar u \in \iso(M).$ So
	$\langle xu \mid u \in U_{i+1} \rangle \subseteq \iso(M)$ and consequently $U_{i+1} \subseteq \iso(M).$
	The same argument can be applied in the second case.
	So we have proved $M_G \leq \iso(M).$ Since $x \in \Omega(G),$ it follows from \ref{inte}
	that $x \in \iso(M).$ But then, since $M_G \leq \iso(M),$ 
	$M=\langle x, y, M_G\rangle = \iso(M)\langle y \rangle,$ hence $M \in \fff$ by Lemma \ref{Icyc}, in contradiction with $g \in \Omega_\ff(M).$
\end{proof}

\begin{lemma}\label{fpf}
	Let $H$ be a 2-generated finite group and $V$ a faithful irreducible $H$-module. Consider the semidirect product $X=V^u \rtimes H,$ with $u=\dim_{\kndo_H(V)}(V).$ 
	Suppose $H=\langle h_1, h_2\rangle.$ Then there exists $w \in V^u \rtimes H$ such that $X=\langle h_1, h_2w\rangle$ if and only
	 if $h_1$ acts fixed-point-freely on $V$.
	\end{lemma}
\begin{proof}
Let $F:=\kndo_H(V).$
We may identify $H =
\langle h_1, h_2 \rangle$ with a subgroup of the general linear group $\gl(u,F)$. In this
identification, for $i=1,2$, $h_i$ becomes an $u\times u$ matrix $X_i$ with coefficients in $F$. Denote by $A_i$ the matrix
$I_u-X_i.$ Let $w=(v_{1},\dots,v_{u})\in V^u.$
Then every $v_{j}$ can be
viewed as a $1 \times u$ matrix. Denote  the $u \times u$ matrix
with rows $v_{1},\dots,v_{u}$  by $B$. By
\cite[Section 4]{CL4}, $\langle h_1, wh_2\rangle=X$ if and only if
\begin{equation}\label{matrici}
		\det \begin{pmatrix}A_1&A_2\\
			0&B\end{pmatrix} \neq 0.
\end{equation}
A matrix $B$ so that (\ref{matrici}) is satisfied can be found if and only if $A_1$ is invertible, i.e. if and only if 
$h_1$ acts fixed-point-freely on $V$.
\end{proof}

\begin{thm}\label{main}
Let $\ff$ be a class of finite groups with the following properties:
\begin{enumerate}
	\item All the groups in $\ff$ are soluble.
	\item $\ff$ is closed under taking subgroups.
\end{enumerate}
Suppose that $G$ is a finite group which is minimal with respect to the following properties: $G$ is  $\ff$-semiregular 
and the graph $\gf(G)$ is not  connected.
Then $G$ is soluble and there exists an epimorphism 
$$\pi: G \to  (V_1 \times \dots \times V_t) \rtimes H,$$ where
$H$ is 2-generated and there exists a faithful irreducible $H$-module $V$ with $V_i \cong_H V$ for $1\leq i \leq t$ and $t=1+\dim_{\kndo_H(V)}(V).$ 

Moreover let $\mathcal W$ be the set of the $H$-submodules  of $V_1\times \dots \times V_t$ that are $H$-isomorphic to $V^{t-1}.$
There exists one and only one $W \in \mathcal W$ with the property that
 $M=\pi^{-1}(W \rtimes H) \notin \ff_2$.  If $g_1, g_2 \in G$ and $\langle g_1, g_2 \rangle \notin \ff$, then either $\langle g_1, g_2\rangle \leq M^x$, for some $x\in G,$ or $H$ is cyclic of prime order and $\langle g_1, g_2\rangle \leq \pi^{-1}(V_1 \times \dots \times V_{t}).$

\end{thm}
\begin{proof}
For a finite group $X$, let $d(X)$ be the smallest cardinality of a generating set of $X$.
By the previous lemmas,  $G$ is soluble and $d(G)\geq 3.$ In particular $G$ contains a normal subgroup $N$ with the property that $d(G/N)= 3$ but $d(G/M)\leq 2$ for every normal subgroup $M$ of $G$ properly containing $N.$ By \cite{dvl}
$G/N \cong V^t \rtimes H$, where $V$ is a faithful irreducible $H$-module and $$t=1+\dim_{\kndo_H(V)}(V).$$ Consider
the epimorphism $\pi: G \to V^t\rtimes H$ and let $\mathcal M$ be the set of the maximal subgroups of $G$ containing $N$ and with the property that $M^\pi$ does not contain $V^t.$ If $M \in \mathcal M,$ then $M/N \cong V^{t-1}\rtimes H.$ Let $F:={\kndo_H(V)}.$ The multiplicity of $V$ in $FH/J(FH)$ (where $FH$ denote the group algebra and $J(FH)$ its Jacobson radical) is $t-1$ so it follows from
\cite[Lemma 1]{cgk} that the smallest cardinality of 
a subset of  $V^t$ generating $V^t$ as an $H$-module is
$\lceil t/(t-1)\rceil=2.$ In particular there is no $w \in W$ with the property that $\langle w \rangle_H=V^t,$ denoting with 
$\langle w \rangle_H$ the $H$-submodule of $V^t$ generated by $w$. If $g \in G,$ then
$g^\pi=wh$ for some $w \in V^t$ and $h \in H$, so $g^\pi \leq \langle w , h \rangle \leq \langle w \rangle_H H,$ and therefore there exists $M\in \mathcal M$ with $g\in M.$

Assume that there exists no
$M \in \mathcal M$ containing two adjacent vertices of $\gf(G).$ 
Let $T$ be the preimage of $H$ under $\pi$.
We must have $T > N$ (otherwise if $g_1, g_2$ are adjacent vertices, then $\langle g_1,g_2\rangle N$ is a proper subgroup of $G$ and it is contained in some $M \in \mathcal M).$ In particular
$T^G=G.$
Fix $l\in T$. For any $g \in G,$ we have $g^\pi=wh,$ with
$w \in V^t$ and $h \in H.$ Thus 
$\langle l, g\rangle^\pi=\langle l^\pi, wh\rangle \leq \langle w \rangle_HH.$ Since, $\langle w \rangle_H \neq V^t,$  there exists $M\in \mathcal M$ containing $\langle l, g\rangle.$ In particular $l$ and $g$ are not adjacent in $\gf(G)$ and therefore $l \in \iso(G).$
Thus, $T\leq \iso(G).$ But then $G=T^G=\iso(G),$ a contradiction.

So we may assume that there exists $K\in \mathcal M$ containing two adjacent vertices $g_1,g_2$ of $\gf(G).$ In particular $K \notin \fff.$
Let $\tilde{\mathcal M}=\{M\in \mathcal M \mid M \notin \ff_2\}.$ If $M_1, M_2 \in \tilde{\mathcal M}$ and $(M_1)_G \neq (M_2)_G$ then
if follows from Lemma \ref{duemax}, (b), that $\Gamma_{M_1}=\Gamma_{M_2}.$

  Suppose that there exists $K^* \in \tilde{\mathcal M},$ with $K^*_G\neq K_G$. By the previous paragraph $\Gamma_K=\Gamma_{K^*}.$ We claim that in this case $\Gamma_K$ contains all the vertices of $\gf(G),$ in contradiction with the assumption that $\gf(G)$ is disconnected. 
  Assume that $g$ is a  vertex  of $\gf(G)$.
There exists $\tilde g$ such that
 $X=\langle g, \tilde g \rangle \notin \ff.$ Let $L$ be a a maximal subgroup of $G$ containing $XN.$ If $L \in \mathcal M,$ then $g \in \Gamma_L =  \Gamma_K=\Gamma_{K^*}$ by Lemma \ref{duemax}, (b). 
If $L \notin \mathcal M,$ then,
 up to conjugacy, we may assume  $K^\pi=W_1H$, $(K^*)^\pi= W_2H$, $L^\pi=V^tX$, with $W_1\neq W_2$, $W_1\cong_H W_2 \cong_H V^{t-1}$ and $X \leq H.$
By Lemma \ref{duemax}, (b),
either $\Gamma_L=\Gamma_K=\Gamma_{K^*}$, or
$\langle K\cap L, K^*\cap L\rangle \leq \iso(L).$ In the latter case $\langle (K\cap L)^\pi, (K^*\cap L)^\pi\rangle=\langle
W_1X, W_2X\rangle=
V^tX=L^\pi,$ but then $\iso(L)=L,$ a contradiction.

So from now on, we will assume that all the maximal subgroups in $\tilde{\mathcal M}$ are conjugate to $K$. It is not restrictive to assume $K^\pi=(V_1\times \dots \times V_{t-1})H$. Suppose now that
$H$ is not cyclic of prime order and, by contradiction, that there exist $g_1,g_2 \in G$ such that $\langle g_1, g_2 \rangle \notin \ff$
and $\langle g_1, g_2 \rangle$ is contained in no conjugate of $K.$
There exists a maximal subgroup $J$ of $G$ with $\langle g_1,g_2\rangle \leq J$ and $J^\pi\geq  V_1\times \dots \times V_t.$ We claim that, for any $g\in G,$ $\Gamma_J=\Gamma_{K^g}.$ Suppose that this is false. By Lemma \ref{duemax}, (c) and (d), $J^\pi=V^t C$ where $C$ is a cyclic maximal subgroup of $H$
and $H=CS$ with $S=\soc H.$
Let $\mathcal U$ be the set of the $H$-submodules $U$ of
$V_1\times \dots \times V_t$ with $U\cong_H V$ and $U\not\leq 
V_1\times \dots \times V_{t-1}.$ If $\pi^{-1}(U)\leq \iso(J)$ for every $U \in \mathcal U,$ then $R:=\langle \pi^{-1}(U) \mid U \in \mathcal U\rangle = \pi^{-1}(V_1\times \dots \times V_t)\subseteq \iso(J).$ Since $J/R$ is cyclic, by Lemma \ref{Icyc}, we would have $J\in \ff_2.$ So there exist $U\in \mathcal U$ and $y_1\notin \iso(J)$ with $y_1^\pi=u \in U.$
There exists $y_2\in J$ with $
\langle y_1,y_2\rangle \notin \ff.$ Let $y_2^\pi=wc,$ with $w\in V^t$ and $c$ in $C.$ If $\langle c \rangle \neq C,$ then
let $L$ be a maximal  subgroup of $G$ containing $\pi^{-1}((V_1\times \dots \times V_t)S\langle c\rangle).$ Clearly $\Gamma_L=\Gamma_J$ since $\{y_1,y_2\} \subseteq \Gamma_L \cap \Gamma_J.$ However, by Lemma \ref{duemax} (c),
$\Gamma_L=\Gamma_{K^g}$. Hence $\Gamma_J=\Gamma_{K^g},$ independently on the choice of $g$. 
Thus $\langle c \rangle=C$. There exists a conjugate $\bar c$ of $c$ in $H$ such that $H=\langle c, \bar c \rangle.$ 
If $c$ acts fixed-point-freely on $V,$ then, up to conjugacy, we may assume $w\in U$, but then $\langle y_1,y_2\rangle \leq \pi^{-1}(UC)$ is contained in a maximal subgroup $M \in \mathcal M$, which is not conjugate to $K$ since $U\leq M_G$ but $U\not\leq K_G.$
Finally assume that $c$ does not act fixed-point-freely.
We have $w=w_1+w_2,$ with $w_1 \in U$ and
$w_2 \in  (V_1\times \dots \times V_{t-1}).$
Since $\bar c$, being a conjugate of $c$, does not act fixed-point-freely of $V,$ by Lemma \ref{fpf}, $\bar c$ is an isolated vertex in $(V_1\times \dots \times V_{t-1})H$, so $\langle  w_2c, \bar c\rangle$ is a proper subgroup of $(V_1\times \dots \times V_{t-1})H$ supplementing $V_1\times \dots \times V_{t-1}.$ Thus, there exists an $H$-submodule $W$
of $V_1\times \dots \times V_{t-1}$ such that $W \cong V^{t-2}$ and
$\langle  w_2c, \bar c\rangle \leq WH.$ But then $M=\pi^{-1}(UWH)\ \in \mathcal M,$ $M$ is not conjugate to $K$ and $\langle y_1,y_2\rangle \leq M.$
\end{proof}

\begin{defn}
	Let $\ff$ be a class of finite groups. We say that a finite group $G$ is strongly $\ff$-semiregular if $\iso(X/Y)\leq X/Y$ for every $X\leq G$ and $Y\nn X.$ 
\end{defn}

\begin{lemma}\label{strongly}
	Let $\ff$ be a class of finite groups with the following properties:
	\begin{enumerate}
		\item All the groups in $\ff$ are soluble.
		\item $\ff$ is closed under taking subgroups.
		\item $\ff$ is closed under taking epimorphic images.
	\end{enumerate}
	Suppose that $G$ is  strongly $\ff$-semiregular and $\gf(H)$ is connected for any proper subgroup and any proper quotient
	$H$ of $G.$ 
	If $\gf(G)$ is disconnected, then $G/M\in\ff_2$ for every $1\neq M\nn G$.
\end{lemma}

\begin{proof} Suppose that there exists $1\neq M\nn G$ such that $G/M \notin \ff_2.$
	Let $I/M:=\iso(G/M)\nn G/M$. If $g_1M,g_2M\notin I/M$, then by the hypotheses, they  belong to the same connected component of the graph $\gf(G/M)$. Since, by Lemma \ref{FFF}, a path from $g_1M$ to $g_2M$ in $\gf(G/M)$ can be lifted to a path between $g_1$ and $g_2$ in $\gf(G)$, it follows that $G\setminus I$ is contained in a unique connected component of $\gf(G)$. By Lemma \ref{normale}, either $\gf(G)$ is connected or there exists $g\in G$ \st $\gen{g}\!I=G$. However, in the second case, $\gen{gM}\!I/M=\gen{gM}\!\iso(G/M)=G/M$ and from Lemma \ref{Icyc} it would follow  $G/M\in\fff,$ against the assumption.
\end{proof}

\begin{lemma}\label{ordini}	Let $\ff$ be a class of finite groups with the following properties:
	\begin{enumerate}
	\item All the groups in $\ff$ are soluble.
\item $\ff$ is closed under taking subgroups.
\item $\iso(G)$ is a subgroup of $G$ for every finite group $G.$
\item If $G_1\in \ff_2$ and $|G_2|=|G_1|,$ then $G_2\in \ff_2.$
\end{enumerate}
Then $\gf(G)$ is connected for every finite group $G.$
\end{lemma}

\begin{proof}
	Let $G$ be a counterexample of minimal order. 
According to the statement of Theorem \ref{main},	
		there exists an epimorphism 
	$\pi: G \to  (V_1 \times \dots \times V_t) \rtimes H,$ where
	$H$ is 2-generated and there exists a faithful irreducible $H$-module $V$ with $V_i \cong_H V$ for $1\leq i \leq t$ and $t=1+\dim_{\kndo_H(V)}(V).$ 	Moreover there exists  $W\in \mathcal W$ with $\pi^{-1}(W \rtimes H)\notin \ff_2$, while
	$\pi^{-1}(W^* \rtimes H)\in \ff_2$ whenever $W\neq W^* \in \mathcal W.$ Since $|\pi^{-1}(W \rtimes H)|=|\pi^{-1}(W^* \rtimes H)|$, this contradicts (4).
\end{proof}

\section{Applications}
\begin{prop}
Let $\mathfrak C$ be the class of the finite cyclic groups and let $G$ be a finite group. Denote by $\tilde{\pi}(G)$ be the set of the prime divisors $p$ of $|G|$ with the property that a Sylow $p$-subgroup of $G$ is either cyclic or a generalized quaternion group. Then $\isoc(G) = \prod_{p\in\tilde\pi(G)}\oo_p(Z(G)).$
\end{prop}

\begin{proof}
Suppose  $x \in \isoc(G).$ Then for any $y \in G,$ $\langle x, y \rangle$ is cyclic, so in particular $[x,y]=1.$ Hence
$x \in Z(G).$ Let $|x|=n,$ $p$ a prime divisor of $n$ and $\tilde x = x^{n/p}.$ Choose $P \in \syl_p(G)$ and let $y$ be an element of $P$ of order $p$. Then $\langle \tilde x, y \rangle$ is contained in 
$\langle x, y \rangle$, hence it is cyclic and therefore $\langle  \tilde x \rangle = \langle y \rangle.$ This means that
$\langle \tilde x \rangle$ is the unique minimal subgroup of $P,$ and therefore $P$ is either cyclic or a generalized quaternion group. We have so proved $\isoc(G)\subseteq \prod_p \oo_p(Z(G)).$ Conversely assume $\tilde \pi=\{p_1,\dots,p_t\}$ and let
$x=x_1\cdots x_t$ with $x_i \in \oo_{p_i}(Z(G)).$ Let $y \in G$.  Since $x \in Z(G),$ $H=\langle x, y \rangle$ is an abelian subgroup of $G$. For any prime $p$ dividing $|G|,$ let $Q=\oo_p(\langle y \rangle)$ and $P=\oo_p(H).$ If $p \notin \tilde \pi$ then $Q=P$ is cyclic. If $p \in\tilde\pi,$ then $P$ is an abelian $p$-group with a unique minimal subgroup, so again it is cyclic. It follows that $H$ itself is cyclic and therefore $x \in \isoc(G).$
\end{proof}

\begin{cor}The graph $\Gamma_{\mathfrak C}(G)$ is connected for every finite group $G.$
\end{cor}

\begin{proof}
	Suppose that $G$ is a finite group of minimal order with respect to the property that $\Gamma_{\mathfrak C}(G)$ is not connected. Let $N$ be a minimal normal subgroup of $G$. By the previous proposition, $G$ is strongly $\mathfrak C$-semiregular, hence it follows from Lemma \ref{strongly} that $G/N$ is cyclic. But this implies that $G$ is 2-generated (see \cite{uni}), in contradiction with Lemma \ref{duegen}.
\end{proof}

\begin{prop}
Let $\mathfrak P$ be the class of the finite groups whose order is divisible by at most a unique prime. Then for any finite group $G$ 
the graph $\Gamma_{\mathfrak P}(G)$ is connected.
 \end{prop}
\begin{proof}
We may assume that all the elements of $G$ have prime power order (otherwise $\gf(G)$ contains a universal vertex). In this case, let $H\leq G.$ We have
$\mathcal{I}_{\mathfrak P}(H)=H$ if $H$ is a $p$-group, $\mathcal{I}_{\mathfrak P}(H)=1$ otherwise. In any case $\mathcal{I}_{\mathfrak P}(H)$ is a subgroup of $H$ and therefore $\gf(G)$ is connected by Lemma 	\ref{ordini}.
\end{proof}
\begin{prop}
Let $\mathfrak D$ be the class of the finite groups $G$ with the property that $\pi(G)$ consists of at most two different primes.
\begin{enumerate}
	\item $\mathfrak D$ is 2-recognizable.
	\item For any $g$ in $G$, let $\pi(g)$ be the set of the prime divisors of $|g|.$ If $|\pi(g)|\geq 3$ for some $g \in G$, then $g$ is adjacent to all the other vertices of $\tilde \Gamma_{\mathfrak D}(G)$. If $|\pi(g)|\leq 2$ for every $g\in G,$ then $\isod(G)$ is a subgroup of $G.$
	\item For any finite group $G$ 
	the graph $\Gamma_{\mathfrak D}(G)$ is connected
	\end{enumerate}
\end{prop}

\begin{proof}
(1)	Assume that $G \in \mathfrak D_2$. This implies that 
$|\pi(g)|\leq 2$ for any $g$ in $G$ (if $|\pi(x)| >  2$ for some $x\in G$, then $\langle x, y \rangle \notin \mathfrak D$ for any $y\in G$). Assume that there exists $\tilde g \in G$ with $\pi(\tilde g)=\{p,q\}.$ Since  $G \in \mathfrak D_2,$ $\langle \tilde g, g\rangle \in \mathfrak D$ for all $g\in G$. But then $\pi(g) \subseteq \{p,q\}$ for all $g \in G$ and therefore $\pi(G)=\{p,q\}$ and $G \in   \mathfrak D.$
If $G$ is not soluble, then there exists $H\leq G$ and $K \unlhd H$ such that
$H/K \cong S$ is a finite non-abelian simple group. In particular
$H=\langle h_1,h_2\rangle K$ for some $h_1, h_2\in H.$ But then $S=\langle h_1,h_2\rangle K/K \cong \langle h_1,h_2\rangle/(K\cap \langle h_1,h_2\rangle),$ so $3\leq \pi(S) \leq \pi(\langle h_1,h_2\rangle),$ in contradiction with $\langle h_1,h_2\rangle 
\in \mathfrak D.$  So we may assume that $G$ is a finite soluble group all of whose elements have prime power order. 
By \cite[Theorem 1]{hei}, $\pi(G)\leq 2$, an again we conclude  $G \in   \mathfrak D.$
	
	\noindent (2) Suppose that $G$ is a finite group of minimal order with respect to the properties that $\isod(G)\neq \emptyset$ 
	and $\isod(G)$ is not a subgroup of $G$. 	Let $x,y\in \isod(G)$ such that $xy\notin \isod(G)$. There exists  $g\in G$  such that $\gen{xy, g}\notin \ddd$. Notice that the minimality property of $G$ implies $G=\gen{x,y,g}.$ Let $M$ be a minimal normal subgroup of $G$ and set $I/M:=\isod(G/M)\nn G/M$.  
	Since $G=\gen{x,y,g}$, we have $\gen{gM}I/M=G/M$.	By Lemma \ref{Icyc}, $G/M\in\ddd_2=\ddd$, so there exist two primes $p_1,p_2$ with $\pi(G/M)\subseteq \{p_1,p_2\}.$
	If $\R(G)=1$, then by \cite[Theorem 6.4]{gu}, for every $1\neq h \in G,$ there exists $\tilde h \in G$ such that $\langle h, \tilde h\rangle$ is not soluble. In particular there exists $\tilde x$ such that $\langle x, \tilde x\rangle \notin \mathfrak D,$
	in contradiction with $x \in \isod(G).$ Thus we may assume that $M$ is an elementary abelian $p$-group. We have $\pi(G/M)=\{p_1,p_2\}$ and $p \notin \{p_1, p_2\},$ otherwise	$G\in \mathfrak D.$
	So $G=M \rtimes H,$ with $H \cong G/M.$ If $g \in \isod(G),$ then $|\pi(g)|\leq 1.$ Suppose that $g$ is a non trivial $p_1$-element of $G$. Then there exists a conjugate $H^*$ of $H$ in $G$ and a non-trivial $p_2$-element $g^*$ such that
	$K=\langle g, g^*\rangle \leq H^*.$ If follows from \cite[Theorem 1]{dlt} that there exists $m \in M$ such that
	$\langle g, g^*m\rangle$ is not a complement of $M$ in $MK.$
	This implies that $p\cdot p_1 \cdot p_2$ divides $|\langle g, g^*m\rangle|$ and consequently $g\notin \isod(G).$ For the same reason if $g$ is a non-trivial $p_2$-element, then $g\notin \isod(G).$ Now we have two possibilities: 
	if $H$ contains no elements of order $p_1\cdot p_2,$ then $\isod(G)=M;$ otherwise $\isod(G)=1.$ In any case, $\isod(G)$ is a subgroup of $G$, against our assumption.
	
	\noindent (c) It follows from Lemma \ref{ordini}.
	\end{proof}

\section{Example 1}

The statement of Theorem \ref{main} cannot be improved. In this section we exhibit a class $\ff$ and a finite group $G$ containing a maximal subgroup $M$ with the property that the the connected components of $\gf(G)$ are precisely the subgraphs induced by the conjugates of $M$ in $G$.

Let $Q$ be the quaternion group and $V\cong C_3\times C_3$  a faithful irreducible $Q$-module. Consider the semidirect product
$X=(V_1\times V_2 \times V_3) \rtimes Q$, with $V_i \cong_Q V$ for
$1\leq i\leq 3.$ Let $T=V_1\times V_2\times V_3$ and $\mathcal U=\{U_1,\dots,U_9\}$ be the set of submodules $U\leq_Q T$ with the properties that $T=(V_1\times V_2) \oplus U.$ Fix $9$ different
primes $p_1,\dots,p_9$ with the property that 12 divides $p_i-1$ for $1\leq i \leq 9.$ Then $Y = V^2 \rtimes Q$ has a faithful irreducible action on a vector space of dimension 8 
over the field $\mathbb F_{p_i}$ with $p_i$ elements. So for $1\leq i\leq 9,$ we may consider an irreducible $X$-module, say $W_i$, with
$W_i \cong C_{p_i}^8$ and $C_X(W_i)=U_i.$ Now consider the semidirect product $$G=(W_1\times \dots \times W_9)\rtimes X.$$
Let $W=W_1\times \dots \times W_9.$ Let $\mathcal M$ be the set of the maximal subgroups $M$ with the property that $W\leq M$ but $T\not\leq M.$ The conjugacy classes of maximal subgroups
in $\mathcal M$ are parametrized by the set $\mathcal Z$ of the $Q$-submodules $Z$ of $T$ with $Z\cong_Q V^2.$
Indeed $\mathcal M$ is the disjoint union of the following families:
$$\mathcal M_Z=\{WZQ^t \mid t \in T\}, \text { with }Z \in \mathcal Z.$$ Let $Z^*=V_1\times V_2.$ Since $Z^* \oplus U_i = T,$ it follows that $W_i$ is a faithful irreducible $M/W$-module for any $M \in \mathcal M_{Z^*}$ and $1\leq i \leq 9.$ Suppose that $Z^* \neq Z \in \mathcal Z.$ There exists $i \in \{1,\dots,9\}$ such that $U_i\leq Z.$ This implies that if $M \in \mathcal M_Z$ then $C_M(W_i) > W$ and consequently the maximal subgroups in $\mathcal M_Z$ are not isomorphic to the ones in $\mathcal M_{Z^*}.$ 

Consider now the class $\mathfrak F$ of the finite groups with no subgroup isomorphic to $M=WZ^*Q.$

Notice that $M$ is 2-generated, while $G$ is not. In particular, if $g_1,g_2\in G,$ then there is an edge $(g_1,g_2)\in \gf(G)$ if and only if $\langle g_1, g_2\rangle  = M^t \in \mathcal M_{Z^*},$ for some $t \in V_3.$

For a finite group $Y$, denote by $\Omega(Y)$ the subset of $G$ consisting of the elements $x$ with the property that
$Y=\langle x, y \rangle$ for some $y \in Y.$

Assume  $g=wzh \in \Omega(M^t)$ with $w\in W, z\in Z^*$ and $h \in Q^t$. It must be $h\notin F^t$, being $F:=\frat(Q).$ Moreover it follows from \cite[Proposition 2.2]{nof} that $wzh \in \Omega(M^t)$  if and only if $zh\in \Omega(\Gamma(Z^*Q^t))$
and, since the action of $Q^t$ on $V$ is fixed-point-free, by Lemma \ref{fpf} the condition $h\notin F^t$ is sufficient to ensure $zh\in \Omega(\Gamma(Z^*Q))$. In particular $\iso(M^t)=WZ^*F^t$. If follows
that if $H\leq G,$ then
$$\iso(H)=\begin{cases}
WZ^*F^t& \text {if $H= M^t$ for some $t \in V_3,$}\\H& \text{otherwise}.	\end{cases}$$
The set $\Delta_t$ of the vertices of $\gf(M^t)$ coincides with
$WZ^*(Q^t \setminus F^t).$
 Moreover, since the action of $Q$ on $V$ is fixed-point-free,
 if $t_1\neq t_2$, then $\Delta_{t_1}\cap \Delta_{t_2}=\emptyset.$ 
So the connected components of $\gf(G)$ are
$WZ^*(Q^t \setminus F^t)$, with $t \in V_3.$

\begin{rem}
In the previous example $\gf(G)$ has 9 connected components. 
We may repeat a similar construction starting from a faithful irreducible action of $Q$ on $C_p\times C_p$, where $p$ is an arbitrary odd prime. In this way we may construct a group $G$ with the property that $\gf(G)$ has  $p^2$ connected components.
\end{rem}

\section{Example 2}

In the previous section we have constructed an example of a group
$G$ satisfying the assumptions of Theorem \ref{main}, in which the graph $\gf(G)$ is disconnected and contains a maximal subgroup $M$ with the property that each arc in the graph belong to the subgraph induced by a conjugate of $M$. However, as it is indicated in the statement of the theorem, there is also the possibility that $G$ contains a proper normal subgroup $N$ in such a way 
that each arc in the graph belongs either to the subgraph induced by a conjugate of $M$ or to the subgraph induced by $N$. In this section we exhibit an example in which this second possibility occur. 

Let $A=\langle a_1\rangle \times \langle a_2\rangle$ with $\langle a_1 \rangle \cong \langle a_2\rangle \cong C_8$ and $Q=\langle b_1, b_2\rangle \cong Q_8.$ Both $A$ and $B$ admit an automorphism of order 3. So we may consider the semidirect
product $X=(A\times Q)\rtimes H$ with $H \cong C_3.$ We assume that the map $V_1=A/\frat(A)\to V_2=B/\frat(B)$ sending $a_1\frat(A)$ to $b_1\frat(B)$ and $a_2\frat(A)$ to $b_2\frat(B)$ is an $H$-isomorphism.
The group $X$ has precisely 6 conjugacy classes of maximal subgroups:
\begin{itemize}
	\item $M_0=A \times Q;$
	\item $M_1=(A \times \langle c \rangle) \cong A \rtimes C_6$ with
	$c=b_1^2=b_2^2;$
	\item $M_2=(A^2 \times Q_8)H$;
	\item $M_h=\langle a_1^hb_1, a_2^hb_2\rangle H,$ with $h \in H.$
\end{itemize}
Moreover $QH$ has an irreducible action on $W\cong C_7 \times C_7$ 
with the property that $Q$ acts fixed-point-freely on $W$ while $C_H(W)=Z \cong C_7.$ This action can be extended to $X$, with $A\leq C_X(W).$
Let $G=W \rtimes X.$
Let 
$$B=AH\times C_7,\quad  C = W \rtimes \langle a_1b_1, b_2\rangle.$$
Moreover let $\mathfrak F$ be the class of finite groups with no subgroup isomorphic to $B$ or $C.$

Suppose that
$g_1=w_1x_1, g_2=w_2x_2$ have the property that $Y=\langle g_1,g_2\rangle \notin \mathfrak F.$ We have two possibilities.

\noindent a) $Y$ has a subgroup isomorphic to $B$. It must be $\langle x_1, x_2 \rangle \leq  M_1^q$ for some $q\in Q$ (since $X$ is not 2-generated and the other maximal subgroups of $X$ with order divisible by 3 do not contain an element of order 8 centralizing a non-trivial element of $W$). Moreover $(ZAH)^q$ is the unique subgroup of $WM_1^q$ isomorphic to $B.$

\noindent b) $Y$ has a subgroup isomorphic to $C.$ Let $J=\langle x_1, x_2\rangle$ and let $M$ be a maximal subgroup of $X$ containing $J.$ It cannot be that $M$ is a conjugate of $M_1,$ since all the 2-subgroups of $M_1$ are abelian. It cannot be that $M$ is a conjugate of $M_2,$ since $M_2$ has no element of order $8$. It cannot be that $M$ is a conjugate
of $M_h,$
since all the elements of order 2 in $M_h$ centralize $W$ while $b_2^2$ does not. So $x_1=r_1s_1,$ $x_2=r_2s_2$ with $r_i \in A, s_i \in Q.$ As in the previous case
we want to describe when a subgroup $\tilde C$ of $M$ is isomorphic to $C$. This occurs if and only if $\tilde C=\langle
r_1s_1,r_2s_2\rangle,$ with $\langle s_1, s_2 \rangle=Q,$
$r_1,r_2\in A,$ $\langle r_1,r_2\rangle/(\langle r_1,r_2\rangle \cap A^2)\cong C_2.$

Let $\ff_B$ (resp. $\ff_C$) the class of finite groups with no subgroup isomorphic to $B$ (resp. $C$).
Let now $K$ be a proper subgroup of $G$, and assume $K \notin \ff_2.$ Since $M_1^q \cap M_0=A\times \langle c \rangle$ for any $q \in Q$, $K$ cannot contain both a subgroup isomorphic to $B$ and a subgroup isomorphic to $C$ so we have two possibilities

\noindent a) $(ZAH)^q \leq K \leq WM_1^q$ for some $q \in Q.$ In this case
$\gf(K)=\Gamma_{\ff_B}(K)$ and $\iso(K)=K\cap (W\frat(A)\langle c \rangle).$

\noindent b) $K\leq M_0$. In this case
$\gf(K)=\Gamma_{\ff_C}(K)$ and $\iso(K)=K\cap (WA\langle c \rangle).$

In particular
$$\Omega_B = \cup_g WM_1^g \setminus W\frat(A)\langle c \rangle
\quad \text { and }\quad \Omega_C = WM_0 \setminus WA\langle c \rangle$$
are, respectively,
the set of vertices of $\gf(G)$ that are non isolated in $\gf(WM_1^q)$ for some $q$ in $Q$ and the set of vertices of $\gf(G)$ that are 
 non isolated in $\gf(WM_0).$
Since $\Omega_B \cap \Omega_C=\emptyset,$ the graph $\gf(G)$ is disconnected (and indeed $\Omega_B$ and $\Omega_C$ are the two connected components of the graph). Moreover $$\iso(G)=(G\setminus \Omega_B) \cap (G \setminus \Omega_C)=W(M_0 \cap (\cap_{q\in Q} M_1^q)) \cup W\frat(A)\langle c \rangle=W\frat(A)\langle c \rangle.$$


\begin{thebibliography}{99}
\bibitem{cgk} J. Cossey, K.W. Gruenberg and L. G. Kovács, The presentation rank of a direct product
of finite groups, J. Algebra 28 (1974), 597–603.
\bibitem{CL4} E. Crestani and A. Lucchini, $d$-Wise generation of prosolvable groups,
J. Algebra 369 (2012), no. 2, 59--69.
\bibitem{cl} E. Crestani and A. Lucchini, The generating graph of finite soluble groups, Israel J. Math. 198 (2013), no. 1, 63--74.
\bibitem{dvl} F. Dalla Volta and A. Lucchini,  Finite groups that need more generators than any proper quotient, J. Austral. Math. Soc. Ser. A 64 (1998), no. 1, 82–91.	
\bibitem{dlt} E. Detomi, A. Lucchini,  M. Mariapia, P. Spiga and G. Traustason, Groups satisfying a strong complement property, J. Algebra 535 (2019), 35--52.
\bibitem{dh} K. Doerk and T. Hawkes, Finite soluble groups. De Gruyter Expositions in Mathematics, 4. Walter de Gruyter \& Co., Berlin, 1992.
\bibitem{Ga} Gasch\"utz, W. Zu einem von B. H. und H. Neumann gestellten Problem, Math. Nachr.
(1955), 249--252.

	\bibitem{gu}  R. Guralnick, B. Kunyavskiĭ, E. Plotkin and A. Shalev,  Thompson-like characterizations of the solvable radical, J. Algebra 300 (2006), no. 1, 363--375.
	\bibitem{hei} H. Heineken,
	On groups all of whose elements have prime power order,
	Math. Proc. R. Ir. Acad. 106A (2006), no. 2, 191–198.
		\bibitem{uni} A. Lucchini and F. Menegazzo, Generators for finite groups with a unique minimal normal subgroup, Rend. Sem. Mat. Univ. Padova 98 (1997), 173–191
	\bibitem{symb} A. Lucchini and F. Menegazzo, Computing a set of generators of minimal cardinality in a solvable group, J. Symbolic Comput. 17 (1994), no. 5, 409--420.
	\bibitem{nof} A. Lucchini and D. Nemmi, The non $\mathfrak F$-graph of a finite group, Math. Nachr. to appear.
\end{thebibliography}
\end{document}

\begin{thm}\label{regolari}
	A formation $\ff$ is regular \ifa every finite \gr $G$ which is soluble and strongly critical for $\ff$ has the property that $G/\soc(G)$ is cyclic.
\end{thm}

\begin{lemma}\label{duemax}
	Let $G$ be a finite group. Suppose that $M$ and $L$ are maximal subgroups of $G$ with $M, L \notin \fff$ and that
	the graphs $\gf(L)$ and $\gf(M)$ are connected. Denote by $\Gamma_M$ (resp. $\Gamma_L$) the connected component of
	$\gf(G)$ containing the isolated vertices
	of $\gf(L)$ (resp. $\gf(M)$).
	\begin{enumerate}[label=(\alph*)]
		\item 
		If $L_G\not\leq M_G,$ then either $\Gamma_L=\Gamma_M$ or $L\cap M \subseteq \iso(L).$ 
		\item If $M_G\not\leq L_G$ and $L_G\not\leq M_G,$ then $\Gamma_L=\Gamma_M.$
	\end{enumerate}
\end{lemma}
\begin{proof} (a)
	If $\Gamma_M \neq \Gamma_L,$ then $L\cap M \subseteq \iso(M) \cup \iso(L).$ Thus either $L\cap M \leq \iso(L)$ or $L\cap M \leq \iso(M)$.
	In the second case, by Theorem \ref{masu},
	$L\cap M$ is a a maximal subgroup of $M,$
	hence $M=(L\cap M)\langle m \rangle$ for some $m \in M.$ But then $M=\iso(M)\langle m\rangle.$ By Lemma \ref{Icyc}, $M \in \fff,$ against the hypothesis. 
	
	\noindent (b) If $\Gamma_M\neq \Gamma_L,$ then, by (a), $L\cap M \leq\iso(L) \cap \iso(M)$, but then $L, M \in \fff.$
\end{proof}
Moreover $\iso(\tilde B)=\langle a_1^2,a_2^2,Z\rangle$ and $\iso(M_1)=\langle a_1^2,a_2^2,i,W\rangle.$
Moreover it must be $\langle s_1, s_2 \rangle=Q.$ Let $C^*=\langle a_1b_1, b_2\rangle$ and $F=\frat(C^*)$
We claim that in $C$ all the elements outside $FW$ are not isolated in the generating graph. Let $z=ws$ with $w\in W$, $s\in C \setminus F.$ There exists $s^*$ with $\langle s, s^*\rangle=C.$
This $s\in V(C^*)$ and consequently $ws \in V(C).$